\documentclass[11pt,reqno]{amsart}
\usepackage{enumerate}
\usepackage{amsrefs}
\usepackage{amsfonts, amsmath, amssymb, amscd, amsthm, bm, cancel}
\usepackage{url}
\usepackage{graphicx}
\usepackage[%backref=page,
linktocpage=true,colorlinks,citecolor=magenta,linkcolor=blue,urlcolor=magenta]{hyperref}
\usepackage{multicol}
\usepackage{comment}
\usepackage[margin=1in]{geometry}
\newtheorem{thm}{Theorem}[section]
\newtheorem*{thmA}{Theorem A}

 \newtheorem{cor}{Corollary}[section]

  \newtheorem{lem}{Lemma}[section]

 \theoremstyle{definition}

 \theoremstyle{remark}
\newtheorem{rem}{Remark}[section]

 \numberwithin{equation}{section}

\allowdisplaybreaks

\newcommand{\f}{\left(}
\renewcommand{\r}{\right)}

\newcommand{\R}{\mathbb{R}}

\renewcommand{\a}{\alpha}

\renewcommand{\d}{\delta}

\renewcommand{\k}{\kappa}

\newcommand{\s}{\sigma}

\begin{document}
\title{A unified flow approach to smooth $L^p$ Christoffel-Minkowski problem for $p>1$}

\author{RUIJIA ZHANG}
\address{Department of Mathematical Sciences, Tsinghua University, Beijing 100084, P.R. China}
\email{\href{mailto:zhangrj17@mails.tsinghua.edu.cn}{zhangrj17@tsinghua.org.cn}}

%\date{\today}

%%% ----------------------------------------------------------------------

\begin{abstract}
In this paper we study an anisotropic expanding flow of smooth, closed, uniformly convex hypersurfaces in $\mathbb{R}^{n+1}$ with speed $\psi\s_k(\lambda)^{\a}$, where $\a$ is a positive constant, $\s_k(\lambda)$ is the $k$-th elementary symmetric polynomial of the principal radii of curvature and $\psi$ is a preassigned positive smooth function defined on $\mathbb{S}^n$. We prove that under some assumptions of $\psi$, the solution to the flow after normalisation exists for all time and converges smoothly to a solution of the well-known $L^p$ Christoffel-Minkowski problem $u^{1-p}\f x \r \s_k \f \nabla^2u+uI\r=c\psi(x)$ for $p>1$.
\end{abstract}

\maketitle
%\tableofcontents

%%%%%%%%
\section{introduction}\label{sec:1}
\subsection{An anisotropic expanding flow} 
Expanding curvature flows have been extensively studied in the past decades.  Let $\mathcal{M}_0$ be a smooth, closed and uniformly convex hypersurfaces in $\mathbb{R}^{n+1}$ enclosing the origin. We consider the flow
\begin{equation}\label{ger1}
\left\{\begin{aligned}
\frac{\partial}{\partial t}X(x,t)=& F^{-\a}(\k)\nu(x,t),\quad \a>0 \\
X(x,0)=& X_0(x)\in \mathcal{M}_0,
\end{aligned}\right.
\end{equation}
where the curvature function $F$ is symmetric, homogeneous of degree 1, strictly monotone and concave. The solution of this class of flows exists on the maximal time $t\in[0,T^*)$(when $0<\a\leq 1$, $T^*=\infty$ and when $\a>1$, $T^*$ is finite) and converges to a sphere after a proper rescaling (see Gerhardt\cites{Ger90, Ger14} and Urbas\cite{U91}). When $\alpha=1$, Urbas \cite{U91} and Gerhardt \cite{Ger90} studied the expanding flows of convex hypersurfaces respectively. Especially, when $0<\a\leq1$, the convex assumption of the initial hypersurface is further weakened in \cite{Ger14}. Here $F$ can be chosen as symmetric curvature function $\f \frac{\s_n}{\s_{n-k}}\r ^{\frac{1}{k}}(\k)$ (see \cite{And07}), where $\s_i$ is the $i$-th elementary symmetric polynomial of principal curvatures. If we further denote $\lambda$ as the principal radii of curvature, then flow \eqref{ger1} can be written as
\begin{equation}\label{ger2}
\partial_t X=\s_k(\lambda)^{\a}\nu,\quad \a>0.
\end{equation}

Generally, we consider the following anisotropic curvature flow in this paper. Let $X:M^n \times [0,T) \rightarrow\mathbb{R}^{n+1}$ be a family of smooth, closed and uniformly convex hypersurfaces, which satisfies
\begin{equation}\label{flow-s}
\left\{\begin{aligned}
\frac{\partial}{\partial t}X(x,t)=& \psi \f \nu \f x,t \r \r \s_k(\lambda ( x,t ) ) ^{\a} \nu(x,t), \\
X(x,0)=& X_0(x),
\end{aligned}\right.
\end{equation}
where $\a>0$, $1 \leq k<n$, $\nu$ is the unit outward normal vector at $X(\cdot,t)$ and $\psi$ is a positive smooth function defined on $\mathbb{S}^n$.

Assume $\mathcal{M}$ is strictly convex. The support function $u$ is defined on $\mathbb{S}^n$ as $u(x)=\max\limits_{X\in\mathcal{M}}\langle X,x\rangle, \ \forall x\in\mathbb{S}^n$. Let $f_i=\nabla_i f$, $f_{ij}=\nabla^2_{ij} f$ and $f_{ijk}=\nabla^3_{ijk} f$, where $\nabla$ is the Levi-Civita connection  corresponding to the standard metric $g$ on $\mathbb{S}^n$. If we choose the local orthonormal frame on $\mathbb{S}^n$, then the principal radii of curvature can be represented as eigenvalues of the matrix $$b_{ij}=u_{ij}+u\delta_{ij}.$$

Analogous to the calculations in \cite{SY20}, we first derive the normalised flow of \eqref{flow-s}. Define $\tilde{X}(t) = \f \frac{|S^n|}{V_{k+1}(u,u...,u)} \r^\frac{1}{k+1} X(t)$ and 
\begin{align*}
\tau=\int_0^t\f \frac{|S^n|}{V_{k+1}(u,u,\cdots,u)}\r^{\frac{1-k\a}{k+1}}ds.
\end{align*}
We have
\begin{align*}
\frac{\partial \tilde{X}}{\partial \tau} = \psi \f \nu \r \s_{k} \f \lambda \r ^{\a} \nu- X(x,t)\int_{\mathbb{S}^n}\hspace{-1.55em}-\ \psi \s_{k} \f \lambda \r  ^{1+\a} .
\end{align*}
If no confusion arises, we substitute $t,X$ for $\tau,\tilde{X}$ and call
\begin{equation}\label{s1:flow-n}
\left\{\begin{aligned}
\frac{\partial}{\partial t}X(x,t)=& \psi \f \nu \f x,t \r \r \s_{k} \f \lambda \r ^{\a} \nu(x,t)- X(x,t)\int_{\mathbb{S}^n}\hspace{-1.55em}-\ \psi \s_{k} \f \lambda \r  ^{1+\a} d\mu_{S^n} ,%\quad k=1,\cdots,n,\\
\\
X(x,0) =&X_0(x),
\end{aligned}\right.
\end{equation}
the normalised flow of \eqref{flow-s}. Clearly, along the flow \eqref{s1:flow-n}, we have
\begin{equation}\label{normalised volume}
\int_{\mathbb{S}^{n}}u\s_k d\mu_{S^n}\equiv |S^n|.
\end{equation}
Flow \eqref{s1:flow-n} can be re-written as a parabolic scalar equation of the support function $u$:
\begin{equation}\label{u flow}
\left\{\begin{aligned}
\frac{\partial}{\partial t}u(x,t)=& \psi \f x,t \r \s_{k} \f \nabla^2 u+uI \r ^{\a} - u(x,t)\int_{\mathbb{S}^n}\hspace{-1.55em}-\ \psi \s_{k} \f \lambda \r  ^{1+\a} d\mu_{S^n} ,%\quad k=1,\cdots,n, on \mathbb{S}^n\\ 
\\
u(x,0) =&u_0(x),
\end{aligned}\right.
\end{equation}
Here we denote $\eta(t)=\int_{\mathbb{S}^n}\hspace{-1.60em}-\ \ \psi \s_{k} \f \lambda \r  ^{1+\a} d\mu_{S^n}$. Later, we will prove flow \eqref{s1:flow-n} converges to the solution of
\begin{align}\label{lp cm}
u^{1-p}\f x \r \s_k \f \nabla^2u+uI\r=c\tilde{\psi}(x) \quad {\rm on} \ \mathbb{S}^n,
\end{align}
where $c$ is a positive constant, $p=1+\frac{1}{\a}$ and $\tilde{\psi}=\psi^{-\frac{1}{\a}}$. Note that the elliptic equation \eqref{lp cm} arises naturally in the Brunn-Minkowski theory. When $k=n,p=1$, it is the well-known Minkowski problem. It concerns about the existence, uniqueness, regularities and stabilities of the prescribed Gauss curvature of hypersurfaces. Minkowski\cite{M03}, Alexandov\cites{A39,A42}, Lewy\cite{L38}, Nirenberg\cite{N53}, Calabi\cite{C58}, Pogorelov\cites{P52,P71}, Cheng-Yau\cite{CY76} etc. made great contributions to the development of this problem. In 1962, Firey \cite{F62} introduced p-sums for convex body. Later, Lutwak \cite{L93} studied the Minkowski problem for Firey's $p$-sum (known as $L^p$ Minkowski problem). $L^p$ Minkowski problem has attracted many researchers in the past years. One can refer to \cites{CW06,LO95, LW13, Zhu15} for more background and comprehensive lists of results. 
When $p=1$ and $1\leq k\leq n-1$, \eqref{lp cm} is the classical Christoffel-Minkowski problem
\begin{equation}\label{cm}
     \s_k \f \nabla^2u+uI\r=\tilde{\psi}(x).
\end{equation}
Guan and Ma \cite{GM03} established constant rank theorem to prove the existence of Eq. \eqref{cm} under the spherical convex assumption of $\tilde{\psi}^{-\frac{1}{k}}$, i.e. $\nabla^2\tilde{\psi}^{-\frac{1}{k}}+\tilde{\psi}^{-\frac{1}{k}}g\geq 0$. $L^p$ Christoffel-Minkowski problem arises naturally in the Christoffel-Minkowski problem related to $p$-sum. It aims to find convex body $\mathcal{K}\in\mathcal{K}_0^{n+1}$(the set of convex bodies containing the origin in their interiors) with prescribed $k$-th $p$-area measure. In smooth category, this class of problem reduces to the problem of finding the convex solutions to the nonlinear elliptic equation \eqref{lp cm}. Hu, Ma and Shen \cite{HMS04} generalized the constant rank theorem in \cite{GM03} to derive the existence of Eq. \eqref{lp cm} for $p\geq k+1$ under the sufficient condition
\begin{align}\label{con1}
\nabla^2\tilde{\psi}^{-\frac{1}{k+p-1}}+\tilde{\psi}^{-\frac{1}{k+p-1}}g\geq 0.
\end{align} 
For $1<p<k+1$, Guan and Xia \cite{GX18} proved the existence results of \eqref{lp cm} under the evenness assumption and condition \eqref{con1} using the constant rank theorem. At the end of \cite{GX18}, they proposed a question: Is there a direct effective estimate of the principal radii of curvature from below without use of the constant rank theorem? Ivaki \cite{Iva19} partially answer this question by using the parabolic approach to avoid employing the constant rank theorem. He obtained the convergence results of the following flow 
\begin{equation}
\frac{\partial}{\partial t}X= \psi \f \nu  \r u^{2-p}\s_k(\lambda) ^{\a} \nu
\end{equation}
and proved the existence results of Eq. \eqref{lp cm} for $p\geq k+1$ under the assumption, $\nabla^2\psi^{-\frac{1}{k+p-1}}+\psi^{-\frac{1}{k+p-1}}g>0$. Furthermore, Bryan, Ivaki and Scheuer \cite{BIS21} provided a parabolic proof to the existence results of Christoffel-Minkowski problem \cite{GM03} by using a constrained curvature flow. It still remained a confusing problem whether this problem can be settled with a parabolic approach for $1<p< k+1$. In this paper, we provide a proof to Guan-Xia's \cite{GX18} results by using an expanding curvature flow without using the constant rank theorem. We partially answer the first question in \cite{GX18}.

We mainly get the following result.
\begin{thm}\label{main-thm-1}
Let $X_0: M^n \to \mathbb \R^{n+1} $ be a smooth, closed, uniformly convex and origin-symmetric hypersurface in $\mathbb R^{n+1}$. Assume $1\leq k\leq  n-1$ and $\a> \frac{1}{k}$, $\psi$ is a positive smooth even function on $\mathbb{S}^n$ and $\nabla^2 \psi^{\frac{1}{1+k\a}}+\psi^{\frac{1}{1+k\a}}g$ is positive-definite. Then the flow \eqref{flow-s} has a unique smooth, uniformly convex solution for finite existence maximal time $t\in [0,T^*)$. After a proper normalisation, the solution of the flow exists for all time and the rescaled hypersurfaces $\tilde{M_t}$ converge to the unique smooth solution of \eqref{lp cm} in the $C^\infty$ topology.
\end{thm}
\begin{rem}
 Ding-Li \cite{DL22} considered the flow \eqref{flow-s} for cases $\a\leq\frac{1}{k}$ or $k=n$.
\end{rem}
The homothetic self-similar solution of the normalised flow \eqref{s1:flow-n} is exactly \eqref{lp cm}, thus we have the following corollary.

\begin{cor}\cite{GX18}\label{cor1}
Let $1\leq k\leq n-1$ and $1<p<k+1$. For any positive even smooth function $\psi$ satisfying $\nabla^2 \tilde{\psi}^{-\frac{1}{k+p-1}}+\tilde{\psi}^{-\frac{1}{k+p-1}}g>0$, \eqref{lp cm} has a unique smooth strictly convex and origin-symmetric solution.
\end{cor}

    \begin{rem}
    Theorem \ref{main-thm-1} provides a new proof to Guan-Xia's existence results of the $L^p$ Christoffel-Minkowski problem for $1<p<k+1$ without using the constant rank theorem.
    \end{rem}
    \begin{rem}
    By using the same argument in \cite{BIS21}, Corollary \ref{cor1} implies the existence results in the weakly convex assumption, i.e. $\nabla^2 \tilde{\psi}^{-\frac{1}{k+p-1}}+\tilde{\psi}^{-\frac{1}{k+p-1}}g \geq 0$.
    \end{rem}

Incidentally, when $\psi\equiv 1$, by using Chow-Gulliver's gradient estimate (see \cites{CG07, Ger14}), Gerhardt obtained the following results 

\begin{thmA}\cite{Ger14}
	Let $M_0$ be a smooth, closed, uniformly convex hypersurface in $\mathbb{R}^{n+1}$ enclosing the origin. If $\psi\equiv 1$, $\a>\frac{1}{k}$, flow \eqref{flow-s} has a unique smooth, uniformly convex solution $M_t$ for the finite maximal time $t\in [0,T)$. After a proper normalisation, $\tilde{M_t}$ converge smoothly to a sphere.
\end{thmA}

The flow \eqref{flow-s} for $1\leq k\leq  n-1$ and $\a> \frac{1}{k}$ ($1<p<k+1$) we study here is more difficult than the cases $\a\leq \frac{1}{k}$ ($p\geq k+1$) or $k=n$ for the following reasons. First, the scalar equation of the flow \eqref{u flow} no longer implies $C^0$ estimates by the maximum principle directly (see \cites{Iva19,SY20}). Second, there is a lack of the uniform lower bounds of vol($\Omega_t$) under the flow \eqref{s1:flow-n}. We choose a different flow from \cites{Iva19,SY20} and  concerns the $L^p$ Christoffel-Minkowski problem for a larger range of $p$. Motivated by Guan-Xia's refined gradient estimate in \cite{GX18}, we obtain a delicate gradient estimate of the flow \eqref{s1:flow-n}. We apply a new method to obtain the uniform lower bound of the support function $u$. We find that using the same technique we can also improve the results of Ivaki \cite{Iva19} to $p>2$ under the evenness assumption. Hence the gradient estimate made in our paper also partially answer the first question in \cite{Iva19}.

This paper is organized as follows. In section \ref{sec:2}, we collect some properties of convex body, show the monotone quantity and derive the evolution equations of geometric quantities and the bounds of $\eta(t)$ along the flow \eqref{s1:flow-n}. In section \ref{sec:3}, we obtain the refined gradient estimate and the uniform bounds of the support function $u$ for the normalised flow \eqref{s1:flow-n}. In section \ref{sec:4}, we derive the uniform bounds of principal curvatures and the long time existence of the flow \eqref{s1:flow-n}. In section \ref{sec:5}, we complete the proof of Theorem \ref{main-thm-1}.

\section{Preliminaries}\label{sec:2}
For convenience, we denote $\eta(t)=\int_{\mathbb{S}^n}\hspace{-1.55em}-\ \ \psi \s_{k} \f \lambda \r  ^{1+\a} d\mu_{S^n}$ and $F=\s_k^{\a}$. Then \eqref{u flow} becomes 
\begin{align}\label{u1:flow-n}
\partial_tu =\psi F -\eta(t) u.
\end{align}
We also name the speed along the normal direction $\psi \s_k^{\a}$ as $\Phi$ in this paper. Now using this scalar equation, we have the following lemma, 
\begin{lem}\label{mon dec}
	Under the normalised flow \eqref{s1:flow-n}, we have
	\begin{align}\label{Ju}
 \mathcal{J}(u)=\int_{\mathbb{S}^n}\psi^{-\frac{1}{\a}}u^{1+\frac{1}{\a}}d\mu_{S^n}
\end{align}
is non-increasing and the equality holds if and only if $\mathcal{M}_t$ satisfies the elliptic equation
\begin{equation}\label{ellip}
\psi u^{-1}\s_k^{\a}\equiv C.
\end{equation} 
\end{lem}
\begin{proof}
\begin{equation}\label{mon dec 2}
\begin{aligned}
\partial_t\mathcal{J}(u)=&(1+\frac{1}{\a})\int_{\mathbb{S}^n}\psi^{-\frac{1}{\a}}u^{\frac{1}{\a}}\partial_t u d\mu_{S^n}\\
=&(1+\frac{1}{\a})\int_{\mathbb{S}^n}\psi^{1-\frac{1}{\a}}u^{\frac{1}{\a}}\s_k^{\a} d\mu_{S^n}-\frac{\int_{\mathbb{S}^n} \psi \s_{k} \f \lambda \r  ^{1+\a} d\mu_{S^n}}{|S^n|}(1+\frac{1}{\a})\int_{\mathbb{S}^n}\psi^{-\frac{1}{\a}}u^{1+\frac{1}{\a}} d\mu_{S^n}\\
\end{aligned}
\end{equation}
Now let $\rho=\frac{\psi \s_k^{\a}}{u}$ and $d\sigma=u\s_kd\mu_{S^n}$. Plugging \eqref{normalised volume} into \eqref{mon dec 2}, we have

\begin{equation}\label{mon dec 3}
\begin{aligned}
\partial_t\mathcal{J}(u)
=(1+\frac{1}{\a})\int_{\mathbb{S}^n}\rho^{1-\frac{1}{\a}}d\sigma-(1+\frac{1}{\a})\frac{\int_{\mathbb{S}^n} \rho d\sigma}{\int_{\mathbb{S}^n}d\sigma}\int_{\mathbb{S}^n}\rho^{-\frac{1}{\a}}d\sigma\leq 0,
\end{aligned}
\end{equation}
where we use Andrew's generalized H$\ddot{\text{o}}$lder inequality (see \cite{And98}) and the identity obtained if and only if $\rho\equiv c$.
\end{proof}
\begin{cor}\label{upp}
Let $M_t$, $t\in[0,T)$, be an smooth, uniformly convex and origin-symmetric solution to the flow \eqref{s1:flow-n}. Assume that $1\leq k\leq n$, $\a>\frac{1}{k}$ and $\psi$ is an even positive function on $\mathbb{S}^n$. Then the support function of it has an upper bound, i.e., there exists a constant $C$ only depends on $\max\psi$ and the initial hypersurface, such that 
\begin{align*}
u\leq C.
\end{align*}
\end{cor}
\begin{proof}
At a fixed time t, assume that u attains it maximum at point $x_0$, i.e. $u_{\max}(t)=u(x_0,t)$. we have $u(x,t)\geq u_{\max}(t)\mid\langle x,x_0\rangle\mid,\forall x\in \mathbb{S}^n$, since u is a even function. Thus,
\begin{align*}
(\max\psi)^{-\frac{1}{\a}}\f u_{\max}(t)\r^{1+\frac{1}{\a}}\int_{\mathbb{S}^n}\mid\langle x,x_0\rangle\mid^{1+\frac{1}{\a}}d\mu_{S^n}\leq\int_{\mathbb{S}^n}\psi^{-\frac{1}{\a}}u^{1+\frac{1}{\a}}d\mu_{S^n}\leq C(M_0).
\end{align*}
where the last inequality comes from Lemma $\ref{mon dec}$. 
Then we obtain $u\leq C$.
\end{proof}

Now we introduce some basic notations in convex geometry. For a function $u\in C^2(\mathbb{S}^n)$, we denote $W_u=\nabla^2u+ug$. Notice that if $u$ is the support function of a strictly convex hypersurface, then $W_u$ is the principal curvature radii of it which is positive-definite. For $\forall A_i\in \mathcal{M}_n$ (the symmetric $n\times n$ matrices), we set
\begin{equation}
\s_k(A_1,A_2,\cdots,A_n)=\frac{1}{k!}\sum^n_{i_1,\cdots,i_n=1\atop j_1,\cdots,j_n=1}\delta^{i_1,i_2,\cdots,i_n}_{j_1,j_2,\cdots,j_n}A_{1_{i_1 j_1}}\cdots A_{k_{i_k j_k}}
\end{equation}
Let $u_i\in C^2(\mathbb{S}^n)$, $i=1,2...,n+1$. Define
\begin{equation*}
V(u_1,u_2,\cdots,u_{n+1}):=\int_{\mathbb{S}^n}u_1\s_n(W_{u_2},\cdots,W_{u_{n+1}})d\mu_{\mathbb{S}^n}
\end{equation*}
\begin{equation*}
V_k(u_1,u_2,\cdots,u_k):=V(u_1,u_2,\cdots,u_k,1,\cdots,1)
\end{equation*}
%and the $k$-convex cone $\Gamma_k$ as 
%\begin{equation}
%\Gamma_k
%\end{equation}
From \cite{GMTZ10}, we have the following properties,
\begin{lem}\cite{GMTZ10}
$V(u_1,u_2,\cdots,u_{n+1})$ is a symmetric multilinear form on $(C^2(\mathbb{S}^n))^{n+1}$. Especially,
\begin{equation*}
V_{k+1}(u,\cdots,u)=V_{k+2}(1,u,\cdots,u)
\end{equation*}
and the Minkowski integral formula holds:
\begin{equation}\label{mink formula}
\int_{\mathbb{S}^n}u\s_k(W_u)d\mu_{S^n}=\frac{k+1}{n-k}\int_{\mathbb{S}^n} \s_{k+1}(W_u)d\mu_{S^n}.
\end{equation}
\end{lem}
The Garding's cone $\Gamma_k$ is given by
\begin{equation*}
\Gamma_k=\lbrace A\in\mathcal{M}_n:\s_i(A)>0 \text{ for } i=1,\cdots,k\rbrace
\end{equation*}
We also have the following Aleksandrov-Fenchel inequality 
\begin{lem}\cite{GMTZ10}
Let $u_i\in C^2(\mathbb{S}^n),i=1,2,\cdots,k$ be such that $u_i>0$ and $W_{u_i}\in\Gamma_k$. Then for any $v\in C^2(\mathbb{S}^n)$, we have
\begin{equation}\label{AF ineq}
V_{k+1}(v,u_1,\cdots,u_k)^2\geq V_{k+1}(v,v,u_2,\cdots,u_k)V_{k+1}(u_1,u_1,u_2,\cdots,u_k).
\end{equation}
The equality holds if and only if $v=au_1+\sum_{l=1}^{n+1}a_lx_l$ for some constants $a,a_1,\cdots,a_{n+1}$.
Especially, there exists a sharp constant $C_{n,k}$ such that,
\begin{equation}\label{mink ineq}
\f\int_{\mathbb{S}^n} \s_{k+1}(\lambda)d\mu_{S^n}\r^{\frac{1}{k+1}}\leq C_{n,k}\f\int_{\mathbb{S}^n} \s_k(\lambda)d\mu_{S^n}\r^{\frac{1}{k}}.
\end{equation}
\end{lem}
  
\subsection{Evolution equations}
For a real symmetric matrix $A =\lbrace b_{ij}\rbrace_{n\times n}$, let $\lambda = (\lambda_1, \lambda_2, \cdots, \lambda_n)$ be the eigenvalues of it and $f(\lambda)$ be a symmetric function of the principal radii of curvature . There exists a function $\mathcal{F}(A)$, such that $F(A)=f(\lambda)$.
We set
\begin{align*}
\dot{ \mathcal{F} }^{pq}(A): =\frac{\partial \hat{\mathcal{F}}}{\partial b_{pq}} \ , \quad  \ddot{\mathcal{F}}^{pq,rs}(A): =\frac{\partial ^2 \hat{\mathcal{F}} }{\partial b_{pq}\partial b_{rs}}.
\end{align*}
and
\begin{align*}
\dot{f}^p(\lambda) :=\frac{\partial f }{\partial \lambda_p}(\lambda)=\dot{\mathcal{F}}^{pp}(A) \ , \quad  \ddot{f}^{pq}(\lambda):=\frac{\partial ^2 f}{\partial \lambda_p\partial \lambda_q}=\ddot{\mathcal{F}}^{pp,qq}(A).
\end{align*}
One can refer to \cite{And07} for more properties of
$\mathcal{F}(A)$.

Denote $\mathcal{L}=\psi\dot{F}^{ij}\nabla^2_{ij}$. We have the following equations (also see \cites{And07, LSW20b, Iva19}).
\begin{lem}
	Along the normalised flow \eqref{s1:flow-n}, we have		\begin{equation} \label{ev-u}
\begin{aligned}
(\partial_t-\mathcal{L})u=(1-k\a)\psi F-\eta u+\psi u\sum\limits_k\dot{F}^{kk},
\end{aligned}
\end{equation}
	\begin{align}\label{bij}
(\partial_t-\mathcal{L})b_{ij}
=\psi_{ij}F+\psi_i F_j+F_i\psi_j+\psi\ddot{F}^{pq,rs}b_{pqi}b_{rsj}+(k\a+1)\psi F\delta_{ij}-\sum\limits_k\dot{F}^{kk}b_{ij}-\eta b_{ij},
\end{align}
\begin{equation} \label{ev-r}
\begin{aligned}
(\partial_t-\mathcal{L})\frac{\rho^2}{2}=F\psi_k u_k-\eta \rho^2-\psi\dot{F}^{ij}b_{ki}b_{kj}+(1+k\a)\psi u F,
\end{aligned}
\end{equation}
	\begin{equation} \label{ev-speed}
\begin{aligned}
(\partial_t-\mathcal{L})\psi\s_k^{\a}=-\eta \psi\dot{F}^{ij}u_{ij}+\psi\dot{F}^{ij}(\psi F-\eta u)\delta_{ij}=- k\a \eta \psi F+\psi^2 F\sum\limits_k\dot{F}^{kk}.
\end{aligned}
\end{equation}	
\end{lem}
\begin{proof}
Using \eqref{u1:flow-n}, we have the evolution of the support function $u$. By direct calculations, 
\begin{equation} 
\begin{aligned}
(\partial_t-\mathcal{L})u&=\psi F-\eta u-\psi \dot{F}^{ij}u_{ij}\\
&=\psi F-\eta u-\psi \dot{F}^{ij}(b_{ij}-u\delta_{ij})\\
&=(1-k\a)\psi F-\eta u+\psi u\dot{F}^{ij}\delta_{ij}
\end{aligned}
\end{equation}
where we use the homogeneity of $F$ in the last equality.
\begin{equation} 
\begin{aligned}
(\partial_t-\mathcal{L})\frac{\rho^2}{2}=&u_{tk}u_k+u\partial_t u-\psi\dot{F}^{ij}u_{ki}u_{kj}-\psi\dot{F}^{ij}u_{kij}u_k-\psi\dot{F}^{ij}u_iu_j-\psi\dot{F}^{ij}u_{ij}u\\
=&F\psi_ku_k+\psi F_k u_k-\eta\mid\nabla u\mid^2-\psi\dot{F}^{ij}(b_{ki}-u\delta_{ki})(b_{kj}-u\delta_{kj})-\psi\dot{F}^{ij}(b_{kij}-u_j\delta_{ki})u_k\\
&+(\psi F-\eta u)u-\psi\dot{F}^{ij}u_i u_j-\psi u \dot{F}^{ij}(b_{ij}-u\delta_{ij})\\
=&F\psi_k u_k-\eta \rho^2-\psi\dot{F}^{ij}b_{ki}b_{kj}+(1+k\a)\psi u F
\end{aligned}
\end{equation}
where we use the Codazzi equation $b_{ikj}=b_{ijk}$ and the homogeneity of $F$ in the last equation.
Next, we calculate the evolution of the speed $\Phi$:
\begin{equation} 
\begin{aligned}
(\partial_t-\mathcal{L})\Phi &=\psi\dot{F}^{ij}\partial_t b_{ij}-\psi\dot{F}^{ij}(\psi F)_{ij}=\psi\dot{F}^{ij}(\partial_t u)_{ij}+\psi\dot{F}^{ij}(\partial_t u)\delta_{ij}-\psi\dot{F}^{ij}(\psi F)_{ij}\\
&=- \psi\eta\dot{F}^{ij}u_{ij}+\psi\dot{F}^{ij}(\psi F-\eta u)\delta_{ij}=-k\a\eta \psi  F+\psi^2 F\dot{F}^{ij}\delta_{ij}
\end{aligned}
\end{equation}
where we use the $k\a$-homogeneity of $F$ in the last equality. At last we calculate the evolution of $b_{ij}$:
\begin{equation}\label{bij1}
\begin{aligned}
(\partial_t-\mathcal{L})b_{ij}&=(\psi F-\eta u)_{ij}+(\psi F-\eta u)\delta_{ij}-\psi\dot{F}^{kl}b_{ijkl}\\
&=\psi_{ij}F+\psi_i F_j+F_i\psi_j+\psi\ddot{F}^{pq,rs}b_{pqi}b_{rsj}+\psi\dot{F}^{kl}b_{klij}-\eta b_{ij}+\psi F\delta_{ij}-\psi\dot{F}^{kl}b_{ijkl}.
\end{aligned}
\end{equation}
Using the Codazzi equation and the Ricci identity, we have
\begin{equation}
\begin{aligned}\label{Ric id}
\dot{F}^{kl}b_{klij}=\dot{F}^{kl}b_{kilj}&=\dot{F}^{kl}b_{kijl}+\dot{F}^{kl}b_{pk}R_{pilj}+\dot{F}^{kl}\delta_{pi}R_{pklj}\\
&=\dot{F}^{kl}b_{ijkl}+\dot{F}^{kl}b_{kl}\delta_{ij}-\dot{F}^{kl}b_{jk}\delta_{il}+\dot{F}^{kl}b_{il}\delta_{kj}-\dot{F}^{kl}b_{ij}\delta_{kl}\\
&=\dot{F}^{kl}b_{ijkl}+k\a F\delta_{ij}-\dot{F}^{kl}\delta_{kl}b_{ij}.
\end{aligned}
\end{equation}
Here we used the fact that $\dot{F}^{ki} b_{kj} = \dot{F}^{kj} b_{ki}$ (see \cite{BIS21} for reference) and $F$ is homogeneous of degree $k \a$. By \eqref{bij1}, we have
\begin{align*}
(\partial_t-\mathcal{L})b_{ij}
=\psi_{ij}F+\psi_i F_j+F_i\psi_j+\psi\ddot{F}^{pq,rs}b_{pqi}b_{rsj}+(k\a+1)\psi F\delta_{ij}-\dot{F}^{kl}\delta_{kl}b_{ij}-\eta b_{ij}.
\end{align*}
Thus, we obtain \eqref{bij}.
%\begin{equation}
%\begin{aligned}\label{h_i^j}
%\end{aligned}
%\end{equation}
%For the first term of \eqref{h_i^j}, we have
%\begin{equation}
%\begin{aligned}\label{Phi_ij}
%\end{aligned}
%\end{equation}
%Using the Codazzi equation, the Ricci identities and $\dot{F}^{pq}h_q{}^l=\dot{F}^{lq}h_q{}^p$, we have
%\begin{equation}
%\begin{aligned}\label{Ric id}
%\dot{F}^{kl}b_{klij}=\dot{F}^{kl}b_{kilj}&=\dot{F}^{kl}b_{kijl}+\dot{F}^{kl}b_{pk}R_{pilj}+\dot{F}^{kl}\delta_{pi}R_{pklj}\\
%&=\dot{F}^{kl}b_{ijkl}+\dot{F}^{kl}b_{kl}\delta_{ij}-\dot{F}^{kl}b_{jk}\delta_{il}+\dot{F}^{kl}b_{il}\delta_{kj}-\dot{F}^{kl}b_{ij}\delta_{kl}\\
%&=\dot{F}^{kl}b_{ijkl}+k\a F\delta_{ij}-\dot{F}^{kl}\delta_{kl}b_{ij}\\
%\end{aligned}
%\end{equation}
%where we used the fact that $\dot{F}^{ki} b_{kj} = \dot{F}^{kj} b_{ki}$ and $F$ is homogeneous of degree $k \a$.
%
%Inserting \eqref{Ric id} and \eqref{Phi_ij} into \eqref{h_i^j}, we have
%\begin{align*}
%\end{align*}
\end{proof}
\subsection{The bounds of $\eta$}
\begin{lem}\label{eta low}
Under the normalised flow \eqref{s1:flow-n}, we have the uniform lower bound of $\eta$, i.e. there exists some $c>0$, such that
\begin{equation*}
\eta\geq c.
\end{equation*}
\end{lem}
\begin{proof}
Combining Minkowski integral formula \eqref{mink formula} and Aleksandrov-Fenchel inequality \eqref{mink ineq}, we obtain by \eqref{normalised volume}
\begin{align*}
|S^n|\equiv \int_{\mathbb{S}^n} u\s_kd\mu_{S^n}=C'_{n,k}\int_{\mathbb{S}^n} \s_{k+1}d\mu_{S^n}\leq C''_{n,k}(\int_{\mathbb{S}^n} \s_kd\mu_{S^n})^{\frac{k+1}{k}}.
\end{align*}
Using H$\ddot{\text{o}}$lder inequality we obtain there exists some constant $c>0$, such that 
\begin{align*}
\int_{\mathbb{S}^n} \psi\s_k^{1+\a}d\mu_{S^n}\geq \min_{\mathbb{S}^n}\psi (\int_{\mathbb{S}^n} \s_kd\mu_{S^n})^{1+\a} \geq c.
\end{align*}

\end{proof}

\begin{lem}\label{eta upp}
Under the normalised flow \eqref{s1:flow-n}, we have the uniform upper bound of $\eta$, i.e. there exists some positive constant C, such that
\begin{equation*}
\eta\leq C.
\end{equation*}
\end{lem}
\begin{proof}
First we calculate the evolution of $\eta(t)$:
\begin{equation}\label{eta app}
\begin{aligned}
\frac{\partial \eta}{\partial t}&=\partial t \frac{\int_{\mathbb{S}^n}\psi\s_k^{1+\a}d\mu_{S^n}}{\mid S^n\mid}=\frac{\int_{\mathbb{S}^n}(1+\a)\psi\s_k^{\a}\dot{\s_k}^{ij}\partial_t b_{ij}d\mu_{S^n}}{\mid S^n\mid}\\
&=(1+\a)\frac{\int_{\mathbb{S}^n}\psi\s_k^{\a}\dot{\s_k}^{ij}\f(\partial_t u)_{ij}+\partial_t u\delta_{ij}\r d\mu_{S^n}}{\mid S^n\mid}\\
&=(1+\a)\frac{\int_{\mathbb{S}^n}\psi\s_k^{\a}\dot{\s_k}^{ij}\f(\psi\s_k^{\a}-\eta(t)u)_{ij}+\f\psi\s_k^{\a}-\eta(t)u\r\delta_{ij}\r d\mu_{S^n}}{\mid S^n\mid}\\
&=(1+\a)\frac{\int_{\mathbb{S}^n}\psi\s_k^{\a}\dot{\s_k}^{ij}\f(\psi\s_k^{\a})_{ij}+\psi\s_k^{\a}\delta_{ij}\r d\mu_{S^n}}{\mid S^n\mid}-k(1+\a)\frac{(\int_{\mathbb{S}^n}\psi\s_k^{1+\a}d\mu_{S^n})^2}{\mid S^n\mid^2}.
\end{aligned}
\end{equation}

By Aleksandrov-Fenchel inequality \eqref{AF ineq}, we have
\begin{align}\label{af2}
V_{k+1}(\psi\s_k^{\a},u,\cdots,u)^2\geq V_{k+1}(\psi\s_k^{\a},\psi\s_k^{\a},u,\cdots,u)V_{k+1}(u,u,\cdots,u).
\end{align}
plugging \eqref{af2} into \eqref{eta app}, we obtain
\begin{align*}
\frac{\partial \eta}{\partial t}\leq k(1+\a)\frac{(\int_{\mathbb{S}^n}\psi\s_k^{1+\a}d\mu_{S^n})^2}{\mid S^n\mid\int_{\mathbb{S}^n}u\s_k d\mu_{S^n}}-k(1+\a)\frac{(\int_{\mathbb{S}^n}\psi\s_k^{1+\a}d\mu_{S^n})^2}{\mid S^n\mid^2}=0.
\end{align*}
Hence $\eta(t)\leq\eta(0)$ and the upper bound of $\eta(t)$ depends on the initial hypersurface.
\end{proof}

\section{The lower bound of the support function}\label{sec:3}

\subsection{A Gradient Estimate}
\begin{lem}\label{c1}
Let $M_t$, $t\in[0,T)$, be an smooth, uniformly convex solution to the normalised flow \eqref{s1:flow-n}. For the case $1\leq k\leq n-1$ and $\a>\frac{1}{k}$, there exists a constant $C>0$ and $\gamma \in (0,1)$, such that the support function $u(\cdot,t)$ on $\mathbb{S}^n \times [0,T)$ satisfies
\begin{equation}\label{gra}
\frac{|\nabla u|^2}{u^\gamma}\leq C
\end{equation}
where $C$, $\gamma$ depends on n, k, $M_0$, $\min \psi$ and  $\|\psi\|_{C^1}$.
\end{lem}
\begin{proof}
Denote $Q$ as $\frac{|\nabla u|^2}{u^\gamma}$. At the maximal point $(x,t)$ of $Q$, we have the critical equation, 
\begin{align}\label{cri}
\frac{|\nabla u|_i^2}{|\nabla u|^2}=\frac{2u_{li}u_l}{|\nabla u|^2}=\gamma \frac{u_i}{u}.
\end{align}
Besides,
\begin{equation}
\begin{aligned}\label{gra1}
0& \leq (\partial_t-\mathcal{L})\log Q\\
& =\frac{2u_{tl}u_l-2\psi\dot{F}^{ij}u_{li}u_{lj}-2\psi\dot{F}^{ij}u_{lij}u_l}{|\nabla u|^2}-\gamma\frac{u_t-\psi\dot{F}^{ij}u_{ij}}{u} +\frac{\psi\dot{F}^{ij}|\nabla u|^2_i|\nabla u|^2_j}{|\nabla u|^4}-\gamma\frac{\psi\dot{F}^{ij}u_iu_j}{u^2}\\
& =\frac{2u_{tl}u_l-2\psi\dot{F}^{ij}u_{li}u_{lj}-2\psi\dot{F}^{ij}u_{lij}u_l}{|\nabla u|^2}-\gamma \frac{u_t-\psi\dot{F}^{ij}u_{ij}}{u} -\f \gamma-\gamma^2 \r\frac{\psi\dot{F}^{ij}u_iu_j}{u^2}.
\end{aligned}
\end{equation}
First we choose the local orthonormal frame ($e_1,e_2,\cdots,e_n$) near $x$. By rotating the coordinate, we further assume $\nabla u= u_1e_1$. From $\eqref{cri}$, we have
\begin{align}\label{u11}
u_{11}=\frac{\gamma}{2}\frac{u_1^2}{u}
\end{align}
and $u_{1i}=0$, $\forall i\neq 1$. Now we can assume ($u_{ij}$) is diagonal, i.e. $u_{ij}=u_{ii}\delta{ij}$. 
Inserting \eqref{u11} into \eqref{gra1}, we obtain
\begin{align*}
0& \leq (\partial_t-\mathcal{L})\log Q\\
 = &\frac{2(\psi F-\eta u)_1}{u_1}-\frac{2\psi\dot{F}^{ii}u_{ii}^2}{|\nabla u|^2}-\frac{2\psi\dot{F}^{ii}u_{1ii}}{u_1}-\gamma \frac{\psi F}{u}+\gamma\eta +\gamma\frac{k\a\psi F}{u} -\gamma\psi \sum_i\dot{F}^{ii} - \f \gamma-\gamma^2 \r\frac{\psi\dot{F}^{11}u_1^2}{u^2}\\
= &\frac{2\psi_1 F}{u_1}+\frac{2\psi F_1}{u_1}-(2-\gamma)\eta-\frac{2\psi\dot{F}^{ii}u_{ii}^2}{|\nabla u|^2}-\frac{2\psi\dot{F}^{ii}\f b_{1ii}-u_i\delta_1^i \r}{u_1}
-\gamma\frac{(1-k\a)\psi F}{u} -\gamma\psi \sum_i\dot{F}^{ii} \\&-2( 1-\gamma)\frac{\psi\dot{F}^{11}u_{11}}{u}.\\
\end{align*}
Here we use \eqref{u11} and $b_{ij}=u_{ij}+u\delta_{ij}$ in the last equality.
% and $\dot{F}^{ii}=\a\s_k^{\a-1}\dot{\s}_k^{ii}>0$
By Codazzi equation, we have $b_{1ii}=b_{ii1}$. Then
\begin{equation}
\begin{aligned}\label{gra rep}
0& \leq (\partial_t-\mathcal{L})\log Q\\
&=\frac{2\psi_1 F}{u_1}-(2-\gamma)\eta-\sum_i\frac{2\psi\dot{F}^{ii}u_{ii}^2}{|\nabla u|^2}+2\psi\dot{F}^{11}-\gamma\psi \sum_i\dot{F}^{ii} -\gamma\frac{(1-k\a)\psi F}{u} - 2( 1-\gamma)\frac{\psi\dot{F}^{11}u_{11}}{u}.
\end{aligned}
\end{equation}
By \eqref{u11} and Corollary \ref{upp}, we have
%$u_{11}=\frac{\gamma}{2}\frac{u_1^2}{u}=\frac{\gamma}{2}Q u^{\gamma-1}\geq CQ u$,

\begin{align*}
2 \frac{\psi \dot{F}^{11} u_{11}^2}{|\nabla u|^2}=2\psi\dot{F}^{11}\frac{\gamma^2}{4}\frac{u_1^2}{u^2}\geq 2\psi\dot{F}^{11}\frac{\gamma^2}{4}Q u^{\gamma-2} \geq CQ\psi\dot{F}^{11}.
\end{align*}
Besides,
\begin{equation}\label{gra first term}
\frac{2\psi_1 F}{u_1}=\frac{2\psi_1 F u^{1-\frac{\gamma}{2}}}{\sqrt{Q}u}\leq C\frac{F}{\sqrt{Q}u}.
\end{equation} 
Then \eqref{gra rep} turns to 
\begin{align}\label{gra2}
0\leq C\frac{F}{\sqrt{Q}u}-\sum_{i\geq2}\frac{2\psi\dot{F}^{ii}u_{ii}^2}{|\nabla u|^2}-\psi\dot{F}^{11}(CQ-2)+\gamma\frac{k\a\psi F}{u} - 2\f 1-\gamma \r\frac{\psi\dot{F}^{11}u_{11}}{u}.
\end{align} 
Here we use the fact $\dot{F}^{ii}=\a\s_k^{\a-1}\dot{\s}_k^{ii}>0$. 
Now we assume $u_{22}\geq u_{33}\geq\cdots\geq u_{nn}$.
Since $u\leq C$ from corollary \ref{upp}, combining with \eqref{cri}, we have
\begin{equation}\label{b11,u}
u_{11}=\frac{\gamma}{2}\frac{u_1^2}{u}=\frac{\gamma}{2}Q u^{\gamma-1}\geq CQ u.
\end{equation}
Thus, $
b_{11}=u_{11}+u\leq (1+\frac{1}{CQ})u_{11}\leq 2u_{11}.
$
Here we assume $Q>\frac{1}{C}$ without loss of generality.

If $u_{11}\geq u_{22}$, the last term in \eqref{gra2} becomes,
\begin{align*}
2(1-\gamma) \frac{\psi\dot{F}^{11}u_{11}}{u}\geq (1-\gamma) \frac{\psi\dot{F}^{11}b_{11}}{u}\geq \frac{\a(1-\gamma)  k\psi F}{n u}
\end{align*}
where we use $\lambda_1\sigma_{k-1}(\lambda|1)\geq \frac{k}{n}\s_k \f \lambda \r$ \cite[pp. 183-184]{Wang09}.
Choose $\gamma=\frac{1}{2n+1}$ and assume $Q$ large enough. We obtain from \eqref{gra2}
\begin{align*}
0\leq \frac{\psi F}{u}\f \frac{C}{\psi\sqrt{Q}}-\frac{k\a}{2n+1}\r.
\end{align*}
Then $Q\leq C(n,\a,k,\min\psi,\|\psi\|_{C_1})$.

If $u_{22}>u_{11}$, then by \eqref{b11,u} we have
\begin{align}\label{b22}
b_{22}=u_{22}+u\leq\f 1+\frac{1}{C Q}\r u_{22}.
\end{align}
Choose $N>1$. \\
Assume $Nu_{11}\leq u_{22}$, 
Then 
\begin{align*}
\frac{2\psi\dot{F}^{22}u_{22}^2}{u_1^2}>\frac{N\gamma\psi\dot{F}^{22}u_{22}}{u}\geq \frac{N\gamma\psi\dot{F}^{22}b_{22}}{(1+\frac{1}{CQ})u}\geq\frac{N\gamma\a k \psi F}{n(1+\frac{1}{CQ}) u}.
\end{align*}
Assume $Nu_{11}>u_{22}$, 
\begin{align}
2(1-\gamma) \frac{\psi\dot{F}^{11}u_{11}}{u}&>2(1-\gamma) \frac{\psi\dot{F}^{22}u_{22}}{Nu}
\geq 2(1-\gamma) \frac{\psi\dot{F}^{22}b_{22}}{(1+\frac{1}{CQ})Nu}
\geq 2\frac{(1-\gamma)k\a\psi  F }{(1+\frac{1}{CQ})Nnu}.
\end{align}
Choose $\delta>\frac{1}{CQ}$, $N=n(1+2\delta)$ and $\gamma=\frac{2}{N^2+2}$, then $\gamma\frac{k\a\psi F}{u}<\min\{\frac{2(1-\gamma)k\a\psi  F }{(1+\frac{1}{CQ})Nnu}, \frac{N\gamma\a k \psi F}{n(1+\frac{1}{CQ}) u}\}$.
The inequality becomes 
\begin{align}
0\leq C\frac{F}{\sqrt{Q}u}-\psi\dot{F}^{11}(CQ-2)-\f \min\lbrace\frac{2(1-\gamma) }{(1+\frac{1}{CQ})Nn}, \frac{N\gamma}{n(1+\frac{1}{CQ})}\rbrace-\gamma\r \frac{k\a\psi F}{u}.
\end{align}
So we have $Q\leq C(n,\a,k,\min\psi,\|\psi\|_{C_1})$.
\end{proof}

\begin{thm}\label{lower bound}
Let $M_t$, $t\in[0,T)$, be a smooth closed uniformly convex and origin-symmetric solution to the flow \eqref{s1:flow-n}. For the case $1\leq k\leq n$ and $\a>\frac{1}{k}$, the support function has the lower bound, i.e. there exists a constant $c>0$, such that
\begin{align*}
u\geq c.
\end{align*}
\end{thm}
\begin{proof}
We prove it by contradiction. First, we assume there exists a sequence $\lbrace t_i\rbrace\rightarrow T$, such that $u\rightarrow 0$. By \eqref{normalised volume}, $u_{\max}$ has a positive lower bound $c$. At a fixed time $t_i$, let $p_{t_i}$, $q_{t_i}\in M_{t_i}$, such that $u_{\min}(t_i)=u(p_{t_i})$, $u_{\max}(t_i)=u(q_{t_i})$. Then we can assume $p_{t_i}$, $q_{t_i}$ in the same quadrant of $R^{n+1}$, since $M^n_{t_i}$ is origin-symmetric. Now we can choose a plane $\mathbb{P}^2_{t_i}$ in $\mathbb{R}^{n+1}$, such that $p_{t_i}$, $q_{t_i}$, the origin $o$ in it and the outer normal vector at $p_{t_i}$ as $y$-axis and $q_{t_i}$ is in the negative $x$-axis direction. Here we denote $q_{t_i}=(q^1_{t_i}, q^2_{t_i})$ (refer to Figure 1). Let $\gamma_{t_i}=\mathbb{P}^{2}_{t_i}\cap M_{t_i}$, $\gamma_{t_i}(0)=q_{t_i}$ and $\gamma_{t_i}(1)=p_{t_i}$. And $\gamma_{t_i}$ is a convex curve, since $M_{t_i}$ is uniformly convex. Denote the included angle between any vector $\overrightarrow{v}$ and $x$-axis as $\a(\overrightarrow{v})$. Then tan$\a\f \overrightarrow{q_{t_i}p_{t_i}}\r \leq\frac{u_{\min}(t_i)}{-q^1_{t_i}}\rightarrow 0$, tan$\a\overrightarrow{\gamma'(0)}=\frac{-q^1_{t_i}}{q^2_{t_i}}\rightarrow\infty$ as $u_{\min}(t_i)\rightarrow 0$ and $u_{\max}(t_i)\geq c$. We can choose $i$ large enough, such that $\a\f\overrightarrow{\gamma_{t_i}'(0)}\r> $ arctan$\frac{2u_{\min}(t_i)}{-q^1_{t_i}}$. Then  there exists $s\in(0,1)$, such that arctan$\overrightarrow{\gamma_{t_i}'(s)}=$ arctan$\frac{2u_{\min}(t_i)}{-q^1_{t_i}}$. For the convexity of $\gamma_{t_i}$, we have $\a \f \overrightarrow{\gamma_{t_i}(0) \gamma_{t_i}(s)} \r > \a \f \overrightarrow{\gamma_{t_i}'(s)} \r$ and tan$\a \f \overrightarrow{\gamma_{t_i}(0) \gamma_{t_i}(s)} \r >$ tan$\a \f \overrightarrow{\gamma_{t_i}'(s)} \r$. By direct computation, we have $\gamma_{t_i}^1(s)-\gamma_{t_i}^1(0)<\frac{-q^1_{t_i}}{2}$. Then $\mid\gamma_{t_i}^1(s)\mid>\frac{-q^1_{t_i}}{2}$ and $\gamma_{t_i}^2(s)<u_{\min}(t_i)$ where $\gamma_{t_i}(s)=\f\gamma_{t_i}^1(s),\gamma_{t_i}^2(s)\r$. Denote $\nu^{p}\f\gamma_{t_i}(s)\r$ as the projection of $\nu\f\gamma_{t_i}(s)\r$ onto $\mathbb{P}^2$. \\
\begin{align*}
u\f\gamma_{t_i}(s)\r=\langle\gamma_{t_i}(s),\nu\f \gamma_{t_i}(s)\r\rangle=\langle\f\gamma_{t_i}^1(s),\gamma_{t_i}^2(s)\r,\nu^{p}\f \gamma_{t_i}(s)\r\rangle
\end{align*}
Then tan $\a\f\nu^{p}\f\gamma_{t_i}\f s\r\r \r=\frac{q^1_{t_i}}{2u_{\min}(t_i)}$ and
\begin{align*}
u \f \gamma_{t_i}(s) \r \leq \gamma_{t_i}^1(s) \cos (\pi-\a(\nu ^{p} ( \gamma_{t_i} (s)))+\gamma_{t_i}^2(s)\leq \gamma_{t_i}^1(s) \cot (-\a(\nu^{p}( \gamma_{t_i} (s)))+\gamma_{t_i}^2(s) \leq 3u_{\min}(t_i).
\end{align*}
Meanwhile, $\mid \gamma_{t_i}(s)\mid> \frac{-q^1_{t_i}}{2}$. Since $\mid \gamma_{t_i}(0) \mid=u_{\max}=\sqrt{(q^1_{t_i})^2+(q^2_{t_i})^2} \geq C$ and $q^2_{t_i}\leq u_{\min} \rightarrow 0$, there exists $i$ large, such that $\mid\gamma_{t_i}(s)\mid> \frac{C}{3}$. By Lemma \ref{c1} and $\mid\gamma_{t_i}(s)\mid^2=\mid\nabla u(\gamma_{t_i}(s)\mid^2+u(\gamma_{t_i}(s))^2\geq \frac{C^2}{9}$. Then $\frac{\mid \gamma_{t_i}(s)\mid}{u^{\frac{\gamma}{2}}\f\gamma_{t_i}(s)\r}\rightarrow\infty$ as $t\rightarrow T$. Thus we get the contradiction.
\end{proof}

\begin{figure}
\caption{$\mathbb{P}^{2}_{t_i}\cap M_{t_i}$}
\includegraphics[width=0.5\textwidth,]{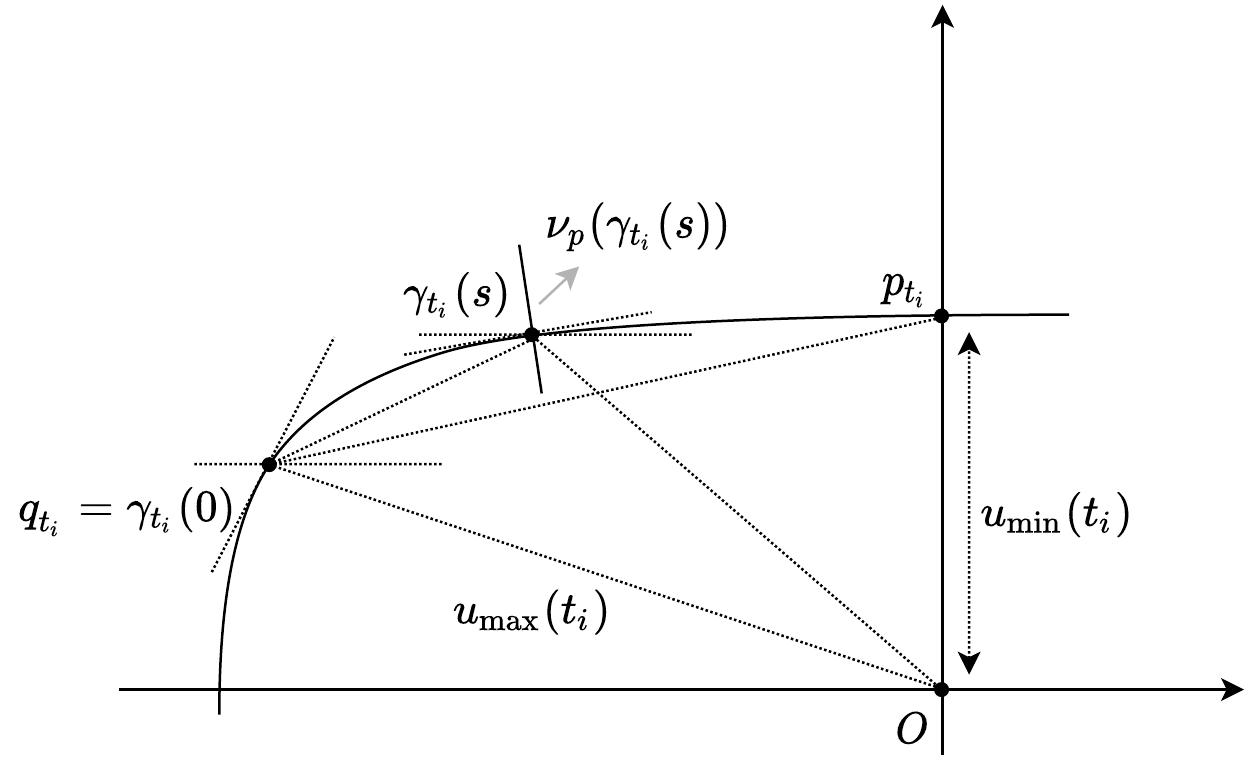}
\end{figure}

\section{Higher regularities}\label{sec:4}

\subsection{bounds of the speed function}
In this section, we will estimate the bounds of $\psi\s_k$ by using the method in \cite{Iva19}.
\begin{lem}\label{Phi l}
Under the normalised flow \eqref{s1:flow-n}, The speed function along the normal direction has a lower bound, i.e.
\begin{align*}
\psi\s_k^{\a}\geq C 
\end{align*}  
\end{lem}
\begin{proof}
Consider an auxiliary function $\Psi=\log(\psi\s_k^{\a})-A\frac{\rho^2}{2}$,  with $A>0$ which is to be chosen later.  By \eqref{ev-r} and \eqref{ev-speed},  we obtain the evolution equation
\begin{equation}\label{low speed1}
\begin{aligned}
(\partial_t-\mathcal{L})\Psi&=\frac{(\partial t-\mathcal{L})\Phi}{\Phi}+\psi\frac{\dot{F}^{ij}\Phi_i\Phi_j}{\Phi^2}-A(\partial_t-\mathcal{L})\frac{\rho^2}{2}\\
&=-\eta k\a +\psi\sum_{i}\dot{f}^{i}+\psi\frac{\dot{f}^{i}\Phi_i^2}{\Phi^2}-A(F\psi_k u_k-\eta \rho^2-\psi\dot{f}^{i}\lambda_i^2+(1+k\a)\psi u F).
\end{aligned}
\end{equation}
%From lemma \ref{sk},  $\dot{F}^{ij}b_{ki}b_{kj}=\a\s_k^{\a-1}\dot{\s_k}^{ij}b_{ki}b_{kj}\geq c \s_k^{\a+\frac{1}{k}}$.
At the minimum point of $\Psi$,  we obtain
\begin{equation}\label{low speed2}
\begin{aligned}
0\geq(\partial_t-\mathcal{L})\Psi\geq\eta(A\frac{\rho^2}{2}-k\a)+A\Phi\f -\frac{\psi_l u_l}{\psi}-(1+k\a)u+\frac{\eta\rho^2}{e^{\Psi+A\frac{\rho^2}{2}}}\r,
\end{aligned}
\end{equation}
where we throw two positive terms $\psi\frac{\dot{f}^{i}\Phi_i^2}{\Phi^2}$ and $A\psi\dot{f}^{i}\lambda_i^2$ in \eqref{low speed1} due to $\dot{f}^i>0$. Now we can choose $A=\frac{3k\a}{u_{\min}^2}$, where $u_{\min}$ is the lower bound obtained in Theorem \ref{lower bound}. We observe that $u_k=\langle X,e_k\rangle\leq \max\Vert X\Vert=u_{\max}$ where $u_{\max}$ is the upper bound of $u$ obtained in Corollary \ref{upp}. Then we have
\begin{equation}
\begin{aligned}
0\geq\eta(A\frac{\rho^2}{2}-k \a)+A\Phi\f-\frac{\psi_l u_l}{\psi}-(1+k\a)u+\frac{\eta\rho^2}{e^{\Psi+A\frac{\rho^2}{2}}}\r\geq 0.
\end{aligned}
\end{equation}
Since we can assume $\Psi$ is far below the zero without loss of generality. Thus, the lower bound of the speed depends on the bounds of $u$ and $\eta$, $\psi_{\min}$ and $\|\psi\|_{C^1}$.
\end{proof}

\begin{lem}\label{Phi}
Under the normalised flow \eqref{s1:flow-n}, The speed function along the normal direction has an upper bound, i.e.
\begin{align*}
\psi\s_k^{\a}\leq C. 
\end{align*}  
\end{lem}
\begin{proof}
First we consider the evolution of $\frac{\Phi}{u}$. By \eqref{ev-u} and \eqref{ev-speed}, we have
\begin{equation}
\begin{aligned}
(\partial_t-\mathcal{L})\frac{\Phi}{u}&=\frac{\Phi}{u}\f\frac{(\partial_t-\mathcal{L})\Phi}{\Phi}-\frac{(\partial_t-\mathcal{L})u}{u}\r+2\frac{\psi\dot{F}^{ij}\Phi_i u_j}{u^2}-2\frac{\psi\Phi\dot{F}^{ij}u_i u_j}{u^3}\\
&=\frac{\Phi}{u}\f(1-k\a)\eta -(1-k\a)\frac{\psi F}{u} \r+2\psi\dot{F}^{ij}(\frac{\Phi}{u})_i \frac{u_j}{u}.
\end{aligned}
\end{equation}
Choose $\epsilon=\frac{1}{ 2u_{\max}^2}$, where $u_{\max}$ is the upper bound in Corollary \ref{upp}.  We assume $\Psi=\frac{\Phi}{u(1-\epsilon\rho^2)}$.  At the maximal point of $\Psi$,
we have the critical equation 
\begin{equation}
\nabla(\frac{\Phi}{u})=-\epsilon\frac{\Phi}{u}\frac{\nabla\rho^2}{1-\epsilon \rho^2}.
\end{equation}
By \eqref{ev-r}, we have
\begin{equation}\label{speed-2}
\begin{aligned}
0&\leq(\partial_t-\mathcal{L})\Psi\\
&=\frac{\Phi}{u(1-\epsilon\rho^2)}(\frac{(\partial_t-\mathcal{L})\frac{\Phi}{u}}{\frac{\Phi}{u}}+\epsilon\frac{(\partial_t-\mathcal{L})\rho^2}{1-\epsilon\rho^2})\\
&=\frac{\Phi}{u(1-\epsilon\rho^2)}\f(1-k\a)\eta -(1-k\a)\frac{\psi F}{u}+2\dot{F}^{ij}(\frac{\Phi}{u})_i \frac{u_j}{F}+\epsilon\frac{F\psi_k u_k-\eta \rho^2-\psi\dot{F}^{ij}b_{ki}b_{kj}+(1+k\a)\psi u F}{1-\epsilon\rho^2}\r.
\end{aligned}
\end{equation}
Here $\dot{F}^{ij}b_{ki}b_{kj}$ can be written as $\sum_i \dot{f}^i\lambda_i{}^2$.
By the formula $\sum_i \frac{\partial \s_k}{\partial \lambda_i} \lambda_i^2 = \s_1 \s_k -(k+1)\s_{k+1}$ and Newton-MacLaurin inequalities, we have
\begin{equation}\label{sk2}
\begin{aligned}
\sum_i\dot{f}^i\lambda_i{}^2=&\sum_i \frac{\partial \s_k^{\a}}{\partial \lambda_i} \lambda_i^2 = \a\s_k^{\a-1}(\s_1\s_k -(k+1) \s_{k+1}) \\
=& \a\s_1\s_k^{\a} - \a(k+1) \s_k^{\a-1} \s_{k+1}\\
\geq& (C_n^k)^{-\frac{1}{k}} \a k \s_k^{\a + \frac{1}{k}}.
\end{aligned}
\end{equation} 
%\begin{align}\label{fij}
%\dot{F}^{ij}b_{ki}b_{kj}=\a\s_k^{\a-1}\dot{\s_k}^{ij}b_{ki}b_{kj}\geq c \s_k^{\a+\frac{1}{k}}
%\end{align}
Direct computation gives 
\begin{equation}\label{cri-speed}
\begin{aligned}
\psi\dot{F}^{ij}(\frac{\Phi}{u})_i u_j&=-\epsilon\psi\frac{\Phi}{u}\frac{\dot{F}^{ij}\rho^2_i u_j}{1-\epsilon\rho^2}\\
&=-\epsilon\psi\frac{\Phi}{u}\frac{\dot{F}^{ij}(2u_i u+2u_{li}u_l)u_j}{1-\epsilon\rho^2}\\
&=-\epsilon\Phi\frac{2\psi\dot{F}^{ii}b_{ii}u_i^2}{u(1-\epsilon\rho^2)}\\
&\leq 0.
\end{aligned}
\end{equation}
Inserting \eqref{sk2} and \eqref{cri-speed} into \eqref{speed-2},
\begin{align*}
0\leq&\frac{\Phi}{u(1-\epsilon\rho^2)} \f (1-k\a) \eta -(1-k\a)\frac{\psi F}{u}+\epsilon \frac{F \psi_k u_k-\eta \rho^2-\psi\dot{F}^{ij}b_{ki}b_{kj}+(1+k\a)\psi u F}{1-\epsilon\rho^2}\r\\
=&\frac{\Phi}{u(1-\epsilon\rho^2)} \f(1-k\a)\eta-\epsilon\frac{\eta\rho^2}{1-\epsilon\rho^2}\r\\
&+\frac{\Phi}{u(1-\epsilon\rho^2)} F\f \frac{\epsilon}{1-\epsilon\rho^2} \f(1+k\a) \psi u+\psi_k u_k \r-(1-k\a) \frac{\psi}{u} \r-\epsilon\frac{\Phi}{u(1-\epsilon\rho^2)}  \frac{\psi}{1-\epsilon\rho^2} \dot{F}^{ij}b_{ki}b_{kj}\\
&\leq -CF^{1+\frac{1}{\a k}}+cF+c.
\end{align*}
Thus F has an upper bound and we complete the proof.
\end{proof}

To obtain the bounds of principal  curvatures, we need the following Lemma (see Urbas \cite{U91}).
\begin{lem}\label{lem-inv-con}
	Denote $\{h^{ij} \}$ as the inverse matrix of $\{b_{ij}\}$, $G=\s_k^{\frac{1}{k}}(b_{ij})$. 
	
Then we have 
	\begin{align}
	(\ddot{G}^{pq,lm} + 2 \dot{G}^{pm}h^{ql})\eta_{pq} \eta_{lm} \geq 2 G^{-1} (\dot{G}^{pq} \eta_{pq})^2\label{f-inv-con}
	\end{align}
	for any tensor $\{\eta_{pq}\}$.
\end{lem}
\begin{proof}
	For convenience, we give a sketch of the proof here. Let $S(h^{ij}) = \frac{1}{G(b_{ij})}$. We have
	\begin{align*}
	\dot{S}^{ij} := \frac{\partial S}{\partial h^{ij}} = - G^{-2} \dot{G}^{pq} \frac{\partial b_{pq}}{ \partial h^{ij}} =
	G^{-2}\dot{G}^{pq}b_{pi} b_{jq}.
	\end{align*}
	Since $S$ is concave, given any tensor $\lbrace\tilde{\eta}^{ij}\rbrace$, we obtain
	\begin{align}
	0 \geq& S^{ij, \lambda \mu} \tilde{\eta}^{ij} \tilde{\eta}^{\lambda\mu} \nonumber\\
	=& \frac{\partial }{\partial b_{lm}}(G^{-2}\dot{G}^{pq}b_{pi} b_{jq}) \frac{\partial b_{lm}}{\partial h^{\lambda\mu}} \tilde{\eta}^{ij} \tilde{\eta}^{\lambda\mu} \nonumber\\
%---------------
=& -b_{l\lambda} b_{\mu m}(-2 G^{-3}\dot{G}^{lm} \dot{G}^{pq}b_{pi} b_{jq} + G^{-2}\ddot{G}^{pq,lm} b_{pi} b_{jq} + G^{-2}\dot{G}^{pq} \d_p{}^l \d_i{}^m b_{jq} \nonumber\\
&+ G^{-2}\dot{G}^{pq}b_{pi} \d_j{}^l \d_q{}^m)\tilde{\eta}^{ij} \tilde{\eta}^{\lambda\mu} \nonumber\\	
%----------
=& 2G^{-3}\dot{G}^{lm}\dot{G}^{pq}b_{pi} b_{jq} b_{l\lambda} b_{\mu m}\tilde{\eta}^{ij} \tilde{\eta}^{\lambda\mu} - G^{-2}\ddot{G}^{pq,lm} b_{pi} b_{jq} b_{l\lambda} b_{\mu m}\tilde{\eta}^{ij} \bar{\eta}^{\lambda\mu}\nonumber\\
&-2G^{-2} \dot{G}^{pq} b_{pi} b_{j\lambda}  b_{\mu q} \tilde{\eta}^{ij} \tilde{\eta}^{\lambda\mu} \nonumber
	\end{align}
where we used $\d_\lambda{}^m= h^{ml} b_{\lambda l}$ in the last equality of the above equation. By setting $\eta_{pm} = b_{pi}\tilde{\eta}^{ij} b_{jm}$, we have
	\begin{align*}
	0 \geq 2 G^{-3} \dot{G}^{pq}\dot{G}^{lm} \eta_{pq} \eta_{lm} - G^{-2} \ddot{G}^{pq,lm}\eta_{pq} \eta_{lm}- 2G^{-2}\dot{G}^{pq} h^{ml} \eta_{pm} \eta_{lq},
	\end{align*}
	which implies the statement.
	\end{proof}
\begin{lem}\label{conv-c2}
	Let $M_t$, $t\in[0,T)$, be a uniformly convex solution to the flow \eqref{s1:flow-n}. Assume $1\leq k\leq n-1$, $\a> 0$, and the positive smooth function $\psi$ on $\mathbb{S}^n$ satisfies $\nabla^2\psi+\psi g>0$. If the support function $u_t$ of the solution satisfies $\frac{1}{C}\leq u\leq C$, there exists a positive constant $C$ depending only on $\a$ and $M_0$, such that the principal curvatures of $X (\cdot, t)$ satisfy
	\begin{align*}
 \frac{1}{C}\leq\lambda_i(\cdot, t) \leq C , \quad \forall \ t \ \in [0,T) \ {\rm and} \ i=1,2,\cdots,n.
	\end{align*}
\end{lem}
\begin{proof}
  Let $\k_{\max}$ be the maximal principal curvature. Denote $Q'=\frac{\k_{\max}}{u}$. Assume $Q'$ attains its maximum at $(x_0,t_0)$. Let $\{h^{ij}\}$ be the inverse matrix of $\{b_{ij}\}$. We choose a normal coordinate system near $x_0\in \mathbb{S}^n$ diagonalizing $\{b_{ij}\}$. Further, after rotating the frame we can assume $\partial_1 |_{(x_0,t_0)}$ is an eigenvector with respect to the minimal eigenvalue $\lambda_{\min}=\frac{1}{\k_{\max}}$. 
 Assume the auxiliary function $Q=\frac{h^{11}}{u}$.
 Clearly, $Q$ attains its maximum at $(x_0,t_0)$. Then
	\begin{equation}\label{par t tilde h}
	\partial_t h^{11}=-( h^{11}) ^2\partial_t b_{11},
	\end{equation}
\begin{align*}
		\nabla_i h^{11} =& \frac{\partial h^{11}}{\partial b_{pq}} \nabla_i b_{pq}
	= - h^{1p} b_{1q} \nabla_i b_{pq}
	= - (h^{11})^2 \nabla_i b_{11}
\end{align*}
 and
\begin{equation}\label{2 deriv tilde h}
	\begin{aligned}
	\nabla_j \nabla_i h^{11} =& \nabla_j (- h^{1p} h^{q1} \nabla_i b_{pq}) \\
	=& -\nabla_j h^{1p} h^{1q} \nabla_i b_{pq} - h^{1p}  \nabla_j h^{1q} \nabla_i b_{pq} - h^{1p} h^{1q} \nabla_j \nabla_i b_{pq}\\
	 =& h^{1r} h^{1p} h^{q1} \nabla_j b{1r}  \nabla_i b_{pq} + h^{1p} h^{r1} h^{qs} \nabla_j b_{sr}\nabla_i b_{pq}- h^{1p} h^{q1} \nabla_j \nabla_i b_{pq} \\
	=& - (h^{11})^2 \nabla_j \nabla_i b_{11} + 2 (h^{11})^2 h^{pl} \nabla_i b_{1p} \nabla_j b_{1l}.
	\end{aligned}
\end{equation}
By using \eqref{par t tilde h} and \eqref{2 deriv tilde h}, we have
\begin{equation}\label{L tilde h 11}
	\begin{aligned}
	(\partial_t-\mathcal{L})h^{11}=&\partial_t h^{11}-\psi \dot{F}^{ij}\nabla_j\nabla_i h^{11}\\
	=& -( h^{11} )^2\partial_t b_{11}+ \psi(h^{11})^2\dot{F}^{ij} \nabla_j \nabla_i b_{11} -2\psi (h^{11})^2  \dot{F}^{ij} h^{pq} \nabla_i b_{1p} \nabla_j b_{1q}\\
	=&-( h^{11} )^2(\partial_t-\mathcal{L} )b_{11}-2 \psi(h^{11})^2 \dot{F}^{ij} h^{pq} \nabla_i b_{1p} \nabla_j b_{1q}.
	\end{aligned}
\end{equation}
Inserting \eqref{bij} into \eqref{L tilde h 11}, we obtain
\begin{equation}\label{h11}
\begin{aligned}
	(\partial_t-\mathcal{L})h^{11}
&=-(h^{11})^2\f\psi_{11}F+2\psi_1 F_1+\psi\ddot{F}^{pq,rs}b_{pq1}b_{rs1}+(k\a+1)\psi F\r+\psi h^{11} \sum\limits_k\dot{F}^{kk}+\eta h^{11}\\
&-2 \psi(h^{11})^2 \dot{F}^{ij} h^{pq} \nabla_i b_{1p} \nabla_j b_{1q}.
\end{aligned}
\end{equation}
Denote $G=\s_k^{\frac{1}{k}}$ and $F=G^{k\a}$. Then 
$$
\dot{F}^{pq}=k\a G^{k\a-1}\dot{G}^{pq},
$$
$$
\ddot{F}^{pq,rs}=k\a G^{k\a-1}\ddot{G}^{pq,rs}+k\a (k\a-1) G^{k\a-2}\dot{G}^{pq}\dot{G}^{rs}.
$$
We calculate at the maximal point of $Q$.  Combining \eqref{ev-u}, we get
\begin{equation}\label{tildeh}	
	\begin{aligned}
	(\partial_t-\mathcal{L})\frac{h^{11}}{u}
&=-\frac{(h^{11})^2}{u}\f\psi_{11}F+2\psi_1 F_1+k\a\psi G^{k\a-1}\ddot{G}^{pq,rs}b_{pq1}b_{rs1}+k\a \psi(k\a-1) G^{k\a-2}(\nabla_1 G)^2+(k\a+1)\psi F\r\\
&+k\a\frac{h^{11}}{u} \psi G^{k\a-1}\sum\limits_k\dot{G}^{kk}+\eta \frac{h^{11}}{u}-2 k\a\psi(h^{11})^2  G^{k\a-1}\dot{G}^{ij} h^{pq} \nabla_i b_{1p} \nabla_j b_{1q}\\
&-(1-k\a)\frac{h^{11}}{u^2}\psi F+\eta \frac{h^{11}}{u}-k\a\psi\frac{h^{11}}{u} G^{k\a-1}\sum\limits_i\dot{G}^{ii}.
	\end{aligned}
\end{equation}	
Substitute $\eta^{ij} $ with $\nabla_1 h^{ij}$ in Lemma \ref{lem-inv-con}. By the Codazzi equation, we have
\begin{equation}\label{use inv-con}
\begin{aligned}
k\a\psi G^{k\a-1}(h^{11})^2\ddot{G}^{pq,rs}b_{pq1}b_{rs1}
+2  k\a\psi G^{k\a-1}(h^{11})^2\dot{G}^{ij} h^{pq} b_{1pi} b_{1qj}
\geq
2k\a \psi G^{k\a-2}(h^{11})^2 \f \nabla_1 G \r^2.
\end{aligned}
\end{equation}	
Inserting \eqref{use inv-con} into \eqref{tildeh}, we obtain
\begin{equation}\label{tildeh2}
	\begin{aligned}
(\partial_t-\mathcal{L})\frac{h^{11}}{u}
\leq&-\frac{(h^{11})^2}{u}\f\psi_{11}F+2\psi_1 F_1+k\a \psi(k\a+1) G^{k\a-2}(\nabla_1 G)^2+(k\a+1)\psi F\r\\
& +2\eta \frac{h_1^{1}}{u}-(1-k\a)\frac{h^{11}}{u^2}\psi F\\
\leq& -\frac{(h^{11})^2}{u}\f\psi_{11}F-\frac{k\a}{k\a+1}F\frac{(\psi_1)^2}{\psi}+(k\a+1)\psi F\r+2\eta \frac{h_1^{1}}{u}-(1-k\a)\frac{h^{11}}{u^2}\psi F,
	\end{aligned}
\end{equation}
where we use the inequality, $2\psi_1 F_1\leq k\a \psi(k\a+1) G^{k\a-2}(\nabla_1 G)^2+\frac{k\a}{k\a+1}F\frac{(\psi_1)^2}{\psi}$.
By using $\nabla_i \nabla_j\psi^{\frac{1}{k\a+1}}+\psi^{\frac{1}{k\a+1}}\delta_{ij}> 0$, we have 
\begin{align}\label{con}
\frac{1}{k\a+1} \nabla_1 \nabla_1\psi+\frac{1}{k\a+1}(\frac{1}{k\a+1}-1)\frac{1}{\psi}(\nabla_1 \psi)^2+\psi > 0.
\end{align}
Inserting \eqref{con} into \eqref{tildeh2}, by Lemma \ref{eta upp}, Corollary \ref{lower bound} and Lemma \ref{Phi}, there exists $c>0$ such that
\begin{align}
(\partial_t-\mathcal{L})\frac{h^{11}}{u}
\leq -c \f\frac{h^{11}}{u}\r^2+c \frac{h_1^{1}}{u}.
	\end{align}

%we have from \eqref{tildeh2}
%\begin{equation}\label{tildeh3}
%	\begin{aligned}
%	\mathcal{L}h^{11}\leq& \f\frac{k}{\b}-1\r\Phi+\gammah^{11}-\frac{\a}{\b}r^{-1}\Phi(h^{11})^2\nabla_1\nabla_1 r-\frac{\a\f \a-\b-k\r }{\b \f \b+k \r } r^{-2}\Phi(h^{11})^2\f \nabla_1r\r ^2\\
%	\leq& \f\frac{k}{\b}-1\r\Phi+\gammah^{11}-\frac{\a}{\b}r^{-1}\Phi(h^{11})^2\nabla_1\nabla_1 r,
%	\end{aligned}
%	\end{equation}
%	where we used  $\dot{G}^{ii} >0$.
%From \eqref{ri} and \eqref{rij}, we have
%\begin{align}\label{hess11 r}
%1- (\nabla_1 r)^2 \geq \frac{u^2}{r^2},
%\quad
%\nabla_1 \nabla_1 r=\frac{1}{r}-\frac{u}{r}h^{11}-\frac{\f \nabla_1 r \r ^2}{r}
%\geq \frac{u^2}{r^3} - \frac{u}{r} h^{11}.
%\end{align}
%Inserting \eqref{hess11 r} into \eqref{tildeh3}, we have
%\begin{align*}
%	\mathcal{L}h^{11}\leq &\f\frac{k}{\b}-1\r\Phi+\f \gamma+\frac{\a}{\b}ur^{-2}\Phi \r h^{11}-\frac{\a}{\b} u^2 r^{-4}\Phi(h^{11})^2.
%\end{align*}
%From Lemma \ref{C0-est}, Lemma \ref{c1} and Corollary \ref{cor Phi}, we derive that %where we derive the maximum  principal radii on $M_{t_0}$,we have
%	\begin{align*}
%	\frac{\partial }{\partial t}h^{11}\leq -c_1(h^{11})^2+c_2h^{11}+c_3.
%\end{align*}
%for some constants $c_1$, $c_2$, $c_3$ at $(x_0,t_0)$.
Hence $h^{11}(x,t)$ has a uniform upper bound, which means that the principal radii are bounded from below by a positive constant $c'$. Thereafter, by Lemma \ref{Phi}, we have
\begin{align*}
C\geq \s_k=&\lambda_{\max}\s_{k-1}(\lambda|\lambda_{\max})+\s_k(\lambda|\lambda_{\max})\\
\geq& C_{n-1}^{k-1}\lambda_{\min}^{k-1}\lambda_{\max}\\
\geq& C_{n-1}^{k-1}(c')^{k-1}\lambda_{\max}
\end{align*}
for some constant $C$.
%Thus there exists a positive constant $C$ such that $\lambda_{\max}\leq C$. It is evident to see that the theorem holds.
Therefore, the principal radii of curvature are bounded from above and below along the normalised flow. This completes the proof of Lemma \ref{conv-c2}.
\end{proof}

 We prove the uniformly parabolic of \eqref{u1:flow-n} in Section \ref{sec:4}, which implies the short time existence of the flow \eqref{s1:flow-n}.  Besides, in Section \ref{sec:3} we obtain the upper lower bounds of $u$ and the gradient estimates along the normalised flow \eqref{s1:flow-n} and the bounds of principal curvatures in Section \ref{sec:4} under the assumptions of Theorem \ref{main-thm-1}. The H$\ddot{\text{o}}$lder estimate of Krylov-Evans \cite{Kryl82} and the parabolic Schauder theory \cite{Lie96} can be applied to derive the higher order derivative estimates of the solution to the flow. Thus, we get the following results
\begin{cor}\label{long time}
Under the assumptions of Theorem \ref{main-thm-1}, the smooth solution of the normalised flow \eqref{s1:flow-n} exists for all time $t\in[0,\infty)$. There exists a constant $C_{\lambda,m}>0$ depending on $\lambda$, $m$, $\a$, $\psi$ and the geometry of $\mathcal{M}_0$, such that 
\begin{align*}
\|u\|_{C^{\lambda,m}}\leq C_{\lambda,m}.
\end{align*}
\end{cor}

\section{Proof of Theorem \ref{main-thm-1}}\label{sec:5}

Due to Lemma \ref{mon dec}, the monotone formula $\mathcal{J}(u)$ is non-increasing, i.e. $\partial_t\mathcal{J}(u)\leq 0$ and the equality holds if and only if $u$ satisfies Eq. \eqref{ellip}. The solution of the flow will smoothly converge to a solution to Eq. \eqref{ellip} in a subsequence due to Corollary \ref{long time}. The full sequence convergence follows from the uniqueness of the $L^p$ Christoffel-Minkowski problem obtained in \cite{HMS04}. Thus we complete the proof of Theorem \ref{main-thm-1}.\qed

\end{document}